\documentclass[12pt]{article}
\usepackage{amsmath}
\usepackage{amsthm}
\usepackage{latexsym}
\usepackage{tikz}
\usetikzlibrary{automata,positioning}
\usetikzlibrary{chains,fit,shapes}
\usetikzlibrary{calc}
\usetikzlibrary{arrows}
\usepackage{tkz-graph}

\newtheorem{theorem}{Theorem}
\newtheorem{definition}[theorem]{Definition}
\newtheorem{lemma}[theorem]{Lemma}
\newtheorem{remark}[theorem]{Remark}
\newtheorem{proposition}[theorem]{Proposition}
\newtheorem{corollary}[theorem]{Corollary}
\newtheorem{conjecture}[theorem]{Conjecture}

\newcommand{\red}{\mathrm{red}}
\newcommand{\comm}[1]{#1}

\newcommand{\clusterXY}{
\begin{tikzpicture}[->,>=stealth',shorten >=1pt,node distance=1cm,auto,main node/.style={rectangle,rounded corners,draw,align=center}]
\node[main node] (1) {$X$}; 
\node[main node] (2) [right of=1] {$Y$}; 
\path [draw,transform canvas={shift={(0,0.1)}}] 
(1) edge node {} (2);
\path [draw,transform canvas={shift={(0,-0.1)}}] 
(1) edge node {} (2);
\end{tikzpicture}
}

\newcommand{\clusterXYtrip}{
\begin{tikzpicture}[->,>=stealth',shorten >=1pt,node distance=1cm,auto,main node/.style={rectangle,rounded corners,draw,align=center}]
\node[main node] (1) {$X$}; 
\node[main node] (2) [right of=1] {$Y$}; 
\path [draw,transform canvas={shift={(0,0.15)}}] 
(1) edge node {} (2);
\path 
(1) edge node {} (2);
\path [draw,transform canvas={shift={(0,-0.15)}}] 
(1) edge node {} (2);
\end{tikzpicture}
}

\title{On shortening u-cycles and u-words for permutations}

\author{Sergey Kitaev\footnote{Department of Computer and Information Sciences,
  University of Strathclyde, 26 Richmond Street,
  Glasgow G1 1XH, United Kingdom,
  \texttt{sergey.kitaev@strath.ac.uk}}, Vladimir N. Potapov\footnote{Sobolev Institute of Mathematics, 4 Acad. Koptyug Ave,
  630090 Novosibirsk, Russia,
  \texttt{vpotapov@math.nsc.ru}}, and Vincent Vajnovszki\footnote{Lib, Universit\'e de Bourgogne Franche-Comt\'e, 
BP 47870, 21078 Dijon Cedex, France, 
  \texttt{vvajnov@u-bourgogne.fr}}}

\begin{document}  

\maketitle

\abstract{This paper initiates the study of shortening universal cycles (u-cycles) and universal words (u-words) for permutations  either by using incomparable elements, or by using non-deterministic symbols. The latter approach is similar in nature to the recent relevant studies for the de Bruijn sequences. A particular result we obtain in this paper is that u-words for $n$-permutations exist of lengths $n!+(1-k)(n-1)$ for $k=0,1,\ldots,(n-2)!$. }  

\section{Introduction}

Chung et al.\ \cite{CDG} introduced the notion of a {\em universal cycle}, or {\em u-cycle}, for permutations, which is a cyclic word such that any permutation of fixed length is order-isomorphic to exactly one factor (that is, to an interval of consecutive elements) in the word. In fact, the notion of a u-cycle for permutations can be extended to that of a u-cycle for any combinatorial class of objects admitting encoding by words \cite{CDG}. In particular, universal cycles for sets of words are nothing else but the celebrated {\em de Bruijn sequences} \cite{CDG}. De Bruijn sequences are a well studied direction in discrete mathematics, and over the years they found widespread use in real-world applications, e.g.\ in the areas of molecular biology \cite{compeau:11}, computer security \cite{MR653429}, computer vision \cite{pscf:05}, robotics \cite{scheinerman:01} and psychology \cite{sbsh:97}.

The existence of u-cycles (of length $n!$) for $n$-permutations (that is, permutations of length $n$) was shown in  \cite{CDG} for any $n$ via {\em clustering} the {\em graph of overlapping $n$-permutations}. This graph has $n!$ vertices labelled by $n$-permutations, and there is an edge $x_1x_2\cdots x_n\rightarrow y_1y_2\cdots y_n$ if and only if the words $x_2x_3\cdots x_n$ and $y_1y_2\cdots y_{n-1}$ are order-isomorphic, that is, if and only if $x_i<x_j$ whenever $y_{i-1}<y_{j-1}$ for all $2\leq i<j\leq n$.

A {\em pattern} of length $k$ is a permutation of $\{1,2,\ldots,k\}$. Each cluster collects all $n$-permutations whose first $n-1$ elements form the same pattern, that is, these elements in each permutation in the cluster  are order-isomorphic to the same $(n-1)$-permutation. We call such a pattern the {\em signature} of a cluster, and we denote a signature by ``$\pi$'' where $\pi$ is an $(n-1)$-permutation. See Figure~\ref{clustering-order-3} for the case of $n=3$, and  Figure~\ref{clustering-order-4} for the case of $n=4$ where clusters are thought of as ``super nodes''.  There is exactly one edge associated with each permutation $x_1x_2\cdots x_n$, which goes to the cluster with the signature that is order-isomorphic to $x_2x_3\cdots x_n$. The edges are also viewed as edges between clusters.

Any Eulerian cycle in a graph formed by clusters can be extended to a Hamiltonian cycle in the graph of overlapping permutations (since each edge corresponds to exactly one permutation and we know this permutation). At least some of these Hamiltonian cycles (possibly all, which is conjectured),  can be extended to u-cycles for permutations via linear extensions of partially ordered sets as described in~\cite{CDG}. 

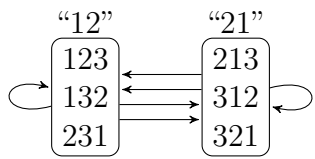
\begin{figure}[h]
\begin{center}
\comm{
\begin{tikzpicture}[->,>=stealth',shorten >=1pt,node distance=2cm,auto,main node/.style={rectangle,rounded corners,draw,align=center}]

\node[main node] (1) {123 \\ 132 \\  231}; 
\node (3) [above of=1,node distance=1cm] {``12''};
\node[main node] (2)  [right of=1] {213 \\ 312 \\  321}; 
\node (4) [above of=2,node distance=1cm] {``21''};

\path[draw,transform canvas={shift={(0,-0.1)}}] 
(1) edge node {} (2);
\path[draw,transform canvas={shift={(0,-0.3)}}] 
(1) edge node {} (2);

\path[draw,transform canvas={shift={(0,0.1)}}] 
(2) edge node {} (1);
\path[draw,transform canvas={shift={(0,0.3)}}] 
(2) edge node {} (1);

\path
(1) edge [loop left] node {} (1);
\path
(2) edge [loop right] node {} (2);

\end{tikzpicture}}
\end{center}
\caption{Clustering the graph of overlapping permutations of order 3}\label{clustering-order-3}
\end{figure}

Removing the requirement for a u-cycle to be a cyclic word, while keeping the other properties, we obtain a {\em universal word}, or {\em u-word}. Of course, existence of a u-cycle $u_1u_2 \cdots u_N$ for $n$-permutations trivially implies existence of the u-word $u_1u_2 \cdots u_Nu_1u_2\cdots u_{n-1}$ for $n$-permutations; the reverse to this statement may not be true.    

\begin{remark}\label{important-remark} {\em It is important to note that {\em any} Hamiltonian path (given by a Hamiltonian cycle) in a graph of overlapping permutations can be easily turned into a u-word for permutations by the methods described in~\cite{CDG}. Indeed, the real problem in the method is dealing with the cyclic nature of a u-cycle making sure that the beginning of it is compatible with the end, while in the case of u-words there are no such complications. As a less relevant observation, note that for the classical de Bruijn sequences, we never have such problems as there is a one-to-one correspondence between Hamiltonian cycles in de Bruijn graphs and de Bruijn sequences.} \end{remark}

In this paper we deal both with the cyclic and non-cyclic cases related to the objects introduced below. This will cause no confusion though as from the context, it will always be clear which case we mean. 

U-cycles and u-words provide an optimal encoding of a set of combinatorial objects in the sense that such an encoding is shortest possible. However, as is discussed in~\cite{CKMS} for the case of {\em de Bruijn sequences}, one can still shorten u-cycles/u-words by using non-deterministic symbols. The studies in \cite{CKMS}, mainly related to binary alphabets, were extended in \cite{Goeckner-et-al} to the case of non-binary alphabets. In this paper, we will utilise the ``shortening'' idea, approaching the problem of shortening u-cycles and u-words for permutations from two different angles discussed next. 
\begin{itemize}
\item Our non-determinism will be in using {\em incomparable elements} and considering {\em linear extensions of partial orders}, and we will study compression possibilities for u-cycles and u-words  for permutations. 
\item Our second approach is a plain extension of the studies in \cite{CKMS,Goeckner-et-al} to the case of permutations. However, using the ``wildcard'' symbol $\Diamond$ seems to be inefficient in the context (it is dominated by non-existence results; see Section~\ref{usage-diamonds}), so we consider its refinement $\Diamond_D$, where $D$ is a subset of the alphabet in question (see Section~\ref{usage-diamond-a-b}).   
\end{itemize}

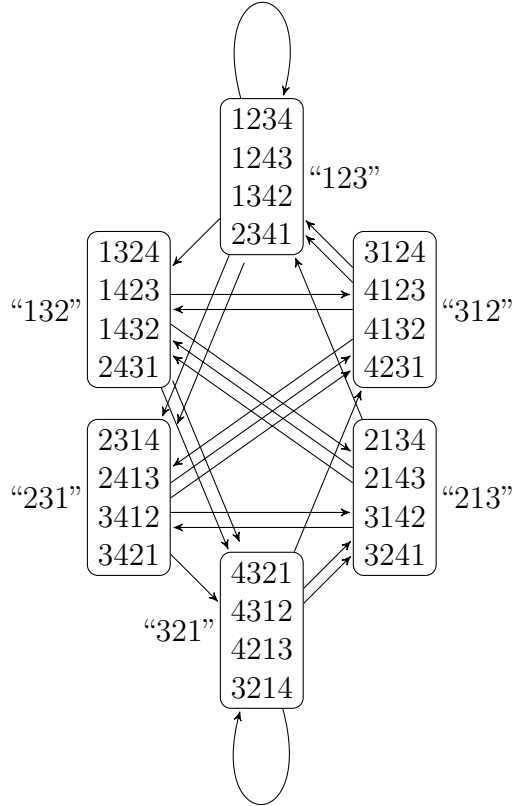
\begin{figure}[h]
\begin{center}
\comm{
\begin{tikzpicture}[->,>=stealth',shorten >=1pt,node distance=2.5cm,auto,main node/.style={rectangle,rounded corners,draw,align=center}]


\node[main node] (1) {1234 \\ 1243 \\ 1342 \\ 2341}; 
\node (7) [right of=1,node distance=1.1cm] {``123''};
\node[main node] (2) [below left of=1] {1324 \\ 1423 \\ 1432 \\ 2431};
\node (10) [left of=2,node distance=1.1cm] {``132''};
\node[main node] (3) [below right of=1] {3124 \\ 4123 \\ 4132 \\ 4231};
\node (8) [right of=3,node distance=1.1cm] {``312''};
\node[main node] (4) [below of=2] {2314 \\ 2413 \\3412 \\ 3421};
\node (11) [left of=4,node distance=1.1cm] {``231''};
\node[main node] (5) [below of=3] {2134 \\ 2143 \\ 3142 \\ 3241};
\node (9) [right of=5,node distance=1.1cm] {``213''};
\node[main node] (6) [below left of=5] {4321 \\ 4312 \\ 4213 \\ 3214};
\node (12) [left of=6,node distance=1.1cm] {``321''};

\path
(1) edge node {} (2)
     edge node {} (4);

\path [draw,transform canvas={shift={(0.2,-0.1)}}] 
(1) edge node {} (4);

\path [draw,transform canvas={shift={(0,0.2)}}] 
(2) edge node {} (3)
     edge node {} (5);

\path [draw,transform canvas={shift={(0.15,0.1)}}] 
(2) edge node {} (6);

\path     
(2) edge node {} (6);

\path [draw,transform canvas={shift={(0,-0.2)}}] 
(3) edge node {} (1);

\path
(3) edge node {} (1)
     edge node {} (2)
     edge node {} (4);

\path [draw,transform canvas={shift={(0,-0.4)}}] 
(4) edge node {} (3);

\path [draw,transform canvas={shift={(0,-0.2)}}] 
(4) edge node {} (3)
     edge node {} (5)
     edge node {} (6);

\path [draw,transform canvas={shift={(0,-0.4)}}] 
(5) edge node {} (4);

\path [draw,transform canvas={shift={(0,-0.2)}}] 
(5) edge node {} (2);

\path
(5) edge node {} (1)
     edge node {} (2);

\path
(6) edge node {} (5)
     edge node {} (3);

\path [draw,transform canvas={shift={(0,-0.2)}}] 
(6) edge node {} (5);
     
\path
(1) edge [loop above] node {} (1);
\path
(6) edge [loop below] node {} (6);

\end{tikzpicture}}
\end{center}
\caption{Clustering the graph of overlapping permutations of order 4}\label{clustering-order-4}
\end{figure}

\subsection{Using linear extensions of partially ordered sets (posets) for shortenning}\label{intro-linear-ext-sec}
To illustrate our idea, consider the word $112$, which is claimed by us to be a u-cycle\footnote{We modify the notion of a u-cycle for $n$-permutations introduced in \cite{CDG}  by allowing equal elements in a factor of length $n$ and declaring them to be incomparable. Note that we still call the obtained object a ``u-cycle for permutations''.} for all permutations of length 3, thus shortening a ``classical'' u-cycle for these permutations, say, $145243$. Indeed, we treat equal elements as {\em incomparable elements}, while the relative order of these incomparable elements to the other elements must be respected. Thus, $112$ encodes all permutations whose last element is the largest one, namely, $123$ and $213$; starting at the second position (and reading the word cyclically), we obtain the word $121$ encoding the permutations $132$ and $231$, and finally, starting at the third position, we (cyclically) read the word $211$ encoding the permutations $312$ and $321$.   More generally, it is clear that the word 
$\underbrace{11\cdots 1}_{n-1\mbox{\tiny\ times}}2=1^{n-1}2$ encodes all permutations and is of length $n$ (instead of length $n!$ for earlier defined u-cycles for permutations). However, there are other compression possibilities creating u-cycles of lengths between $n$ and $n!$.  For example, the word $1232$ is also a u-cycle for permutations of length $3$. Note that the word of the form $11\cdots 1$ is a (trivial) {\em u-word} for all permutations of the respective length (when words are not read cyclically), while this word is not a u-cycle because the definition of a u-cycle cannot be applied to it. 

The main goal of this paper is to study compression possibilities for (classical) u-cycles and u-words for permutations. In particular, we will show that such u-words exist of lengths $n!+(1-k)(n-1)$ for $k=0,1,\ldots,(n-2)!$ (see Theorem~\ref{thm1}) and we conjecture that a similar result is true for u-cycles (see Conjecture~\ref{conj1}). More specifically, our concern will be in existence of u-cycles/u-words for permutations in which equal elements do not stay closer than a fixed number of elements $d\geq 1$ from each other, that is, when there are at least $d-1$ other elements between any pair of equal elements. Note that the case of $d\geq n$  is not interesting when dealing with $n$-permutations since then equal elements cannot appear in the same factor of length $n$, and therefore, such a problem would be equivalent to constructing classical u-cycles/u-words for $n$-permutations, which has already been solved. Thus, the interesting values for $d$ for us are between 1 and $n-1$. 

Finally, note that the problem can be modified by requiring from equal elements to stay {\em exactly}, rather than {\em at least}, at distance $d$, $1\leq d\leq n-1$, from each other, and then one can study the lengths of possible u-cycles/u-words for permutations, if any. Both problems are, of course, equivalent for the case $d=n-1$, which we deal with in Section~\ref{main-gen-res-sec}.

\subsection{Using $\Diamond$s for shortenning}\label{wild-subsec}

In \cite{CKMS,Goeckner-et-al} u-cycles for words (de Bruijn sequences) and u-words for words are shortened using the $\Diamond$ symbol playing the role of a ``wildcard'' symbol, or a ``universal symbol''. Any word containing a $\Diamond$ is called a {\em partial word}, or {\em p-word} in \cite{CKMS,Goeckner-et-al}, and the universal cycles/words obtained by shortening with $\Diamond$s are called, respectively, {\em universal partial cycles}, or {\em u-p-cycles}, and {\em universal partial words}, or {\em u-p-words}. For example, $u=\Diamond\Diamond0111$ is a u-p-word for binary words of length 3, since 
\begin{itemize}
\item $\Diamond\Diamond0$ covers 000, 010, 100 and 110; 
\item $\Diamond01$ covers 001 and 101; and
\item the remaining factors in $u$ cover 011 and 111.
\end{itemize}

As a straightforward extension of the objects in \cite{CKMS,Goeckner-et-al} to the case of permutations, our u-cycles and u-words will contain $\Diamond$(s), whose meaning needs to be redefined though to avoid factors not order-isomorphic to permutations. In analogy with \cite{CKMS,Goeckner-et-al}, we call u-cycles and u-words for permutations containing at least one $\Diamond$ universal partial cycles (u-p-cycles) and universal partial words (u-p-words) for permutations, respectively. Introducing these notions helps us to distinguish between shortening using linear extensions of posets (when the resulting objects are still called by us u-cycles and u-words; see Section~\ref{intro-linear-ext-sec}), and shortening using $\Diamond$s, in which case the obtained objects are called u-p-cycles and u-p-words.

To see which of the $n$-permutations are covered by a factor of length $n$, we keep the same relative order of non-$\Diamond$ elements, and insert all possible elements instead of the $\Diamond$(s) that will result in the reduced form (see Subsection~\ref{definitions} for definitions) in an $n$-permutation. Following this definition, for $n=3$, $1\Diamond2$ covers the permutations 213, 123 and 132, while for $n=4$, $1\Diamond 2\Diamond$ covers the following 12 permutations: 3142, 3241, 2143, 2341, 2134, 2431, 1243, 1342, 1234, 1432, 1324 and 1423. Any factor of length  $n$ with $k$ $\Diamond$s covers $\frac{n!}{(n-k)!}$ permutations. Indeed, the number of ways to pick values for the $\Diamond$s is ${n\choose k}$, and there are $k!$ ways to arrange these values.

We say that a u-p-word for $n$-permutations is {\it trivial} if it contains only $\Diamond$s. Obviously, $\Diamond$ is the only u-p-word for the permutation of length 1. Also, $\Diamond 1$ is a u-p-word for 2-permutations.
Proposition~\ref{diamond-at-first} below shows that if $n\geq 3$ then there is no u-p-word containing a single $\Diamond$ that is placed in position 1. This result, along with Proposition~\ref{prop1} and Corollaries \ref{cor-10} and \ref{nice-corol}, led us to the observation that usage of $\Diamond$s in u-p-cycles, or u-p-words, for permutations may be too restrictive to be of practical use, and instead of a $\Diamond$, one should use a {\em restricted} $\Diamond$ denoted $\Diamond_{D}$, where $D$ is a subset of $\{1, 2,\ldots,n\}$ and $n$ is the size of permutations in question. Indeed, even though no u-p-word for 3-permutations of the form $\Diamond x_1x_2\cdots x_k$ exists by Proposition~\ref{diamond-at-first}, for example, $\Diamond_{1,2}254231$ is a u-p-word for 3-permutations (in particular, the factor $\Diamond_{1,2}25$ covers the permutations 123 and 213). See Theorem~\ref{restr-Diamond} for a result in this direction. 

So, $\Diamond_D$ gives the permissible extensions out of $n$ possible extensions given by $\Diamond$. However, note that the notion of a $\Diamond_D$ is well-defined only if there is {\em at most} one  $\Diamond_D$ in any factor of length $n$, since there is no meaning of, for example, the factor $\Diamond_{1,2}\Diamond_{1,2}\Diamond_{1,2}1$ for $n=4$. Having said that, it is always acceptable to have $\Diamond_{D_1}$, $\Diamond_{D_2},\ldots,\Diamond_{D_k}$ inside the same factor of length $n$ as long as $D_1\cap D_2\cap\cdots\cap D_k=\emptyset$.  

\subsection{Some basic definitions}\label{definitions}

For a word $w=w_1\cdots w_n$ over an ordered alphabet, we let $\red(w)$ denote the word that is obtained from $w$ by replacing each copy of the $i$-th smallest element in $w$ by $i$. For example, $\red(2547)=1324$, $\red(5470)=3241$ and $\red(436326)=324214$.  

Let $\pi$ be a permutation of $\{1,\ldots,n\}$ and $x$ an element of $\{1,\ldots,n\}$. For $x<n$, we let $x^+$ denote a number $y$ such that $x<y<x+1$, while for $x=n$, $x^+=n+1$. Also, for $x>1$, we let $x^-$ denote an element $y$ such that $x-1<y<x$, while for $x=1$, $x^-=0$. The definitions of $x^+$ and $x^-$ can be generalized to any word instead of a permutation $\pi$ in a straightforward way, namely, $x^+$ refers to an element larger than $x$ but less than next largest element (if it exists), while $x^-$ refers to an element smaller than $x$ but larger than next smallest element (if it exists).   

The {\em complement} of an $n$-permutation $\pi_1\pi_2\cdots \pi_n$ is  the permutation obtained by replacing $\pi_i$ by $n+1-\pi_i$. For example, the complement of $2314$ is $3241$. The {\em reverse} of a permutation is the permutation written in the reverse order. For example, the reverse of $2341$ is $1432$.  

\subsection{Organization of the paper} 

This paper is organized as follows. In Section~\ref{lin-ext-sec} we discuss shortening u-cycles and u-words for permutations via linear extensions of posets and present a key result, Theorem~\ref{thm1}, giving possible lengths of u-words for permutations. An extension of the results in Section~\ref{lin-ext-sec} in the case of $n=4$ is discussed in Section~\ref{possible-lengths-u-cycles-sec}. In Section~\ref{sec3} we discuss the usage of $\Diamond$ (see Section~\ref{usage-diamonds}) and $\Diamond_D$ for the special case of $D$ being of size 2 (see Section~\ref{usage-diamond-a-b}) in the context of shortening u-cycles and u-words for permutations.  Finally, in Section~\ref{final-sec} we give some concluding remarks and state some problems for further research.

\section{Shortening u-cycles/u-words for permutations via linear extensions of posets}\label{lin-ext-sec}

In Section~\ref{main-gen-res-sec} we will derive Theorem~\ref{thm1} showing possible lengths of u-words when incomparable elements are allowed at distance $n-1$ for $n$-permutations. In Section~\ref{possible-lengths-u-cycles-sec} we will provide an example for $n=4$ of a shorter u-cycle  than those given by Theorem~\ref{thm1}. The example was obtained by allowing incomparable elements to be closer to each other (to be at distance 2 rather than at distance 3). 

\subsection{Incomparable elements at distance $n-1$ for $n$-permutations}\label{main-gen-res-sec}

\begin{definition} Two different permutations, $\pi_1\cdots\pi_n$ and $\sigma_1\cdots\sigma_n$, are called {\em twin permutations}, or {\em twins}, if 
\begin{itemize}
\item $\red(\pi_1\cdots\pi_{n-1})=\red(\sigma_1\cdots\sigma_{n-1})$, and
\item $|\pi_n-\pi_1|=|\sigma_n-\sigma_1|=1$.   
\end{itemize}\end{definition}
Examples of twins are 3124 and 4123, 2413 and 3412, and 23451 and 13452.

We refer the Reader to Figures~\ref{clustering-order-3} and~\ref{clustering-order-4} to check their understanding of the following four lemmas in the cases of $n=3$ and $n=4$, respectively. 

\begin{lemma}\label{lem1} Each cluster has exactly one pair of twins. \end{lemma}

\begin{proof} Let the signature (the first $n-1$ elements of the permutations in the reduced form) of a cluster  be ``$x_1\cdots x_{n-1}$''. The only possibilities to create twin permutations are to adjoin $x^{+}_1$ or $x^{-}_1$ at the end of $x_1\cdots x_{n-1}$, and these possibilities always exist. \end{proof}

By {\em parallel edges} between clusters we mean multiple edges oriented in the same way. In particular, a pair (resp., a triple) of parallel edges is called a {\em double edge} (resp., a {\em triple edge}). In what follows, double and triple edges from a cluster $X$ to a cluster $Y$ will be denoted, respectively, by \clusterXY and \clusterXYtrip.

\begin{lemma}\label{lem2} For any cluster $X$, there exists a unique cluster $Y$ such that \clusterXY. Also, for no clusters $X$ and $Y$, we have \clusterXYtrip. \end{lemma}

\begin{proof} Both of the statements follow from the fact that parallel edges can only be produced by twins (the last $(n-1)$ elements in non-twin permutations in a cluster cannot be isomorphic), but by Lemma~\ref{lem1}, there is only one such pair in each cluster. \end{proof}

\begin{lemma}\label{lem3}  For any cluster $Y$, there exists a unique cluster $X$ such that \clusterXY. \end{lemma}

\begin{proof} Let the signature of $Y$  be ``$x_1\cdots x_{n-1}$''. Then the only double edge that can come to $Y$ is given by the permutations $x^{-}_{n-1}x_1\cdots x_{n-1}$ and $x^{+}_{n-1}x_1\cdots x_{n-1}$ (both belonging to the same cluster with the signature ``$x_{n-1}x_1\cdots x_{n-2}$'').  \end{proof}

By Lemmas~\ref{lem2} and~\ref{lem3}, the clustered graph of overlapping permutations can be partitioned into a disjoint union of cycles formed by double edges.  

\begin{lemma}\label{lem4} Any of the disjoint cycles formed by the double edges goes through exactly $n-1$ distinct clusters. \end{lemma}

\begin{proof}  Since double edges are formed by twin permutations, we can assume that any such cycle is of the form:

\begin{center}
\begin{tikzpicture}[->,>=stealth',shorten >=1pt,node distance=3.4cm,auto,main node/.style={rectangle,rounded corners,draw,align=center}]
\node[main node] (1) {$x_1x_2\cdots x_{n-1}x^{+}_1$ \\ $x_1x_2\cdots x_{n-1}x^{-}_1$}; 
\node[main node] (2) [right of=1] {$x_2x_3\cdots x_{n-1}x_1x^{+}_2$ \\ $x_2x_3\cdots x_{n-1}x_1x^{-}_2$}; 
\node (3) [right of=2,shift={(-1,0)}] {$\cdots$}; 
\node[main node] (4) [right of=3,shift={(-0.8,0)}] {$x_{n-1}x_1\cdots x_{n-2}x^{+}_{n-1}$ \\ $x_{n-1}x_1\cdots x_{n-2}x^{-}_{n-1}$}; 
\path [draw,transform canvas={shift={(0,0.1)}}] 
(1) edge node {} (2);
\path [draw,transform canvas={shift={(0,-0.1)}}] 
(1) edge node {} (2);
\path [draw,transform canvas={shift={(0,0.1)}}] 
(2) edge node {} (3);
\path [draw,transform canvas={shift={(0,-0.1)}}] 
(2) edge node {} (3);
\path [draw,transform canvas={shift={(0,0.1)}}] 
(3) edge node {} (4);
\path [draw,transform canvas={shift={(0,-0.1)}}] 
(3) edge node {} (4);
\end{tikzpicture}
\end{center}
where the last cluster is linked to the first one by a double edge. Since all $x_i$s are distinct, the cycle must involve exactly $n-1$ clusters.\end{proof}

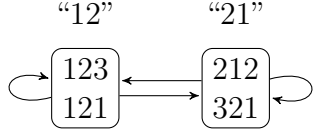
\begin{figure}[h]
\begin{center}
\comm{
\begin{tikzpicture}[->,>=stealth',shorten >=1pt,node distance=2cm,auto,main node/.style={rectangle,rounded corners,draw,align=center}]

\node[main node] (1) {123 \\ 121}; 
\node (3) [above of=1,node distance=1cm] {``12''};
\node[main node] (2)  [right of=1] {212 \\  321}; 
\node (4) [above of=2,node distance=1cm] {``21''};

\path[draw,transform canvas={shift={(0,-0.1)}}] 
(1) edge node {} (2);

\path[draw,transform canvas={shift={(0,0.1)}}] 
(2) edge node {} (1);

\path
(1) edge [loop left] node {} (1);
\path
(2) edge [loop right] node {} (2);

\end{tikzpicture}}
\end{center}
\caption{Applying incomparable elements on distance 2 for 3-permutations}\label{dist-2-order-3}
\end{figure}

\begin{theorem}\label{thm1} Using incomparable elements at distance $n-1$, one can obtain u-words for $n$-permutations of lengths $n!+n-1$, $n!$, $n!-(n-1),\ldots,n!-(n-1)!+n-1$.\end{theorem}

\begin{proof} 
It is not hard to show, and is stated in \cite{CDG}, that the clustered graph of overlapping $n$-permutations is balanced and strongly connected for any $n\geq 1$.

There are $(n-1)!$ clusters. By Lemma~\ref{lem4}, there are $(n-2)!=(n-1)!/(n-1)$ disjoint cycles formed by double edges, and we can decide in which cycles to replace every double edge by a single edge thus maintaining the property of the graph (whose nodes are clusters) being balanced. This action will correspond to replacing every double edge of the form

\begin{center}
\begin{tikzpicture}[->,>=stealth',shorten >=1pt,node distance=3.5cm,auto,main node/.style={rectangle,rounded corners,draw,align=center}]
\node[main node] (1) {$x_1x_2\cdots x_{n-1}x^{+}_1$ \\ $x_1x_2\cdots x_{n-1}x^{-}_1$}; 
\node[main node] (2) [right of=1] {$x_2x_3\cdots x_{n-1}x_1x^{+}_2$ \\ $x_2x_3\cdots x_{n-1}x_1x^{-}_2$}; 
\path [draw,transform canvas={shift={(0,0.1)}}] 
(1) edge node {} (2);
\path [draw,transform canvas={shift={(0,-0.1)}}] 
(1) edge node {} (2);
\end{tikzpicture}
\end{center}

by 

\begin{center}
\begin{tikzpicture}[->,>=stealth',shorten >=1pt,node distance=3.5cm,auto,main node/.style={rectangle,rounded corners,draw,align=center}]
\node[main node] (1) {$x_1x_2\cdots x_{n-1}x_1$}; 
\node[main node] (2) [right of=1] {$x_2x_3\cdots x_{n-1}x_1x_2$}; 
\path  
(1) edge node {} (2);
\end{tikzpicture}
\end{center}
and thus introducing incomparable elements inside some of clusters. Strong connectivity in the graph will clearly be maintained as well, since our action is in simply replacing a pair of equivalent edges by a single edge.

So, by removing double edges in such a way we guarantee the existence of an Eulerian cycle going through the clusters, which gives the existence of the respective Hamiltonian cycle (recall that to each word or permutation there corresponds exactly one edge), and thus the existence of a respective u-word for $n$-permutations by Remark~\ref{important-remark}. 
\end{proof}

See Figure~\ref{dist-2-order-3} for an illustration of the proof of Theorem~\ref{thm1} in the case of $n=2$, and Figure~\ref{dist-3-order-4} for that in the case of $n=3$ when both of the double edge cycles were replaced. Examples of u-words that can be obtained from Figures~\ref{dist-2-order-3} and~\ref{dist-3-order-4}, respectively, are 123212 and 123847687657859423123. 

We believe that replacing ``u-words'' by ``u-cylces'' in Theorem~\ref{thm1}, and adjusting the lengths, would result in a true statement. We state this as the following conjecture. 

\begin{conjecture}\label{conj1} Using incomparable elements at distance $n-1$, one can obtain u-cycles for $n$-permutations of lengths $n!$, $n!-(n-1)$, $n!-2(n-1),\ldots,n!-(n-1)!$. \end{conjecture}

An attempt to solve Conjecture~\ref{conj1} would follow the same steps as those in the proof of Theorem~\ref{thm1}, but an argument would then need to be made that there is a Hamiltonian cycle in the graph of overlapping permutations that can be turned into a u-cycle for permutations. There is a chance that {\em any} Hamiltonian cycle has this property. Examples of u-cycles supporting Conjecture~\ref{conj1}, that can be obtained from Figures~\ref{dist-2-order-3} and~\ref{dist-3-order-4}, respectively, are 1232 and 123847687657859423. 

\begin{figure}[h]
\begin{center}
\comm{
\begin{tikzpicture}[->,>=stealth',shorten >=1pt,node distance=2.5cm,auto,main node/.style={rectangle,rounded corners,draw,align=center}]


\node[main node] (1) {1234 \\ 1243 \\ 1231}; 
\node (7) [right of=1,node distance=1.1cm] {``123''};
\node[main node] (2) [below left of=1] {1324 \\ 1423 \\ 1321};
\node (10) [left of=2,node distance=1.1cm] {``132''};
\node[main node] (3) [below right of=1] {3123 \\ 4132 \\ 4231};
\node (8) [right of=3,node distance=1.1cm] {``312''};
\node[main node] (4) [below of=2] {2314 \\ 2312 \\ 3421};
\node (11) [left of=4,node distance=1.1cm] {``231''};
\node[main node] (5) [below of=3] {2134 \\ 2132 \\ 3241};
\node (9) [right of=5,node distance=1.1cm] {``213''};
\node[main node] (6) [below left of=5] {4321 \\ 4312 \\ 3213};
\node (12) [left of=6,node distance=1.1cm] {``321''};

\path
(1) edge node {} (2)
     edge node {} (4);


\path [draw,transform canvas={shift={(0,0.2)}}] 
(2) edge node {} (3)
     edge node {} (5);

\path     
(2) edge node {} (6);

\path
(3) edge node {} (1)
     edge node {} (2)
     edge node {} (4);

\path [draw,transform canvas={shift={(0,-0.2)}}] 
(4) edge node {} (3)
     edge node {} (5)
     edge node {} (6);

\path [draw,transform canvas={shift={(0,-0.4)}}] 
(5) edge node {} (4);

\path
(5) edge node {} (1)
     edge node {} (2);

\path
(6) edge node {} (5)
     edge node {} (3);
     
\path
(1) edge [loop above] node {} (1);
\path
(6) edge [loop below] node {} (6);

\end{tikzpicture}}
\end{center}
\caption{Applying incomparable elements at distance 3 for 4-permutations}\label{dist-3-order-4}
\end{figure}
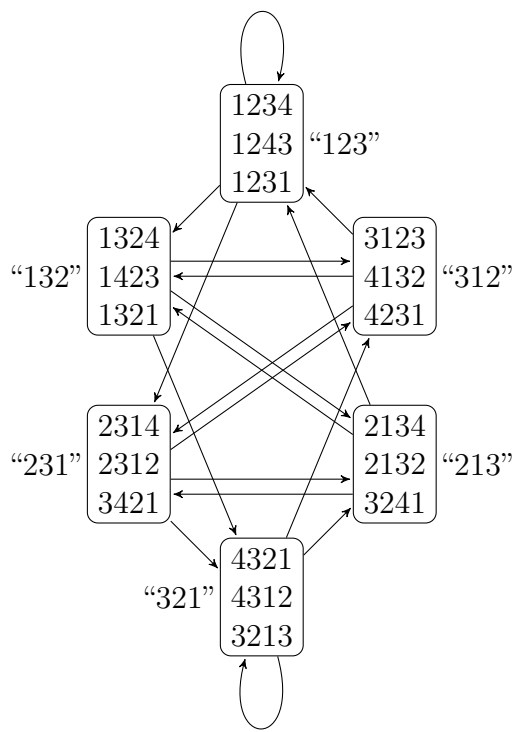

\subsection{Lengths of u-cycle and u-words for $n$-per-mutations different from $n$ and those in Theorem~\ref{thm1} and Conjecture~\ref{conj1}}\label{possible-lengths-u-cycles-sec}

Theorem~\ref{thm1} (resp., Conjecture~\ref{conj1}) discusses a number of (resp., potentially) possible lengths for u-words (resp., u-cycles)  for $n$-permutations, the shortest of which is $n!-(n-1)!+n-1$ (resp., $n!-(n-1)!$). The trivial u-cycle $1^{n-1}2$ and the trivial u-word $1^n$ for $n$-permutations are of length $n$. A natural question is whether there exist  u-words and u-cycles for $n$-permutations of length larger than $n$ but smaller than $n!-(n-1)!+n-1$ and $n!-(n-1)!$, respectively. Clearly, such u-cycles/u-words could only be obtained if incomparable elements would be allowed on the distance smaller than $n-1$. An example of such a u-cycle is $u=34321432345234$ of length 14 for 4-permutations (note that 4!-3!=18 is the shortest u-cycle suggested by Conjecture~\ref{conj1} and actually given in Section~\ref{main-gen-res-sec}). Clustering the graph of overlapping 4-permutations with incomparable elements used is shown in Figure~\ref{dist-2-order-4}. A u-word of length 17 ($<21$, the shortest length in Theorem~\ref{thm1}) is easily obtained  from $u$ by adjoining 343 at the end: 34321432345234343.

\begin{figure}[h]
\begin{center}
\comm{
\begin{tikzpicture}[->,>=stealth',shorten >=1pt,node distance=2cm,auto,main node/.style={rectangle,rounded corners,draw,align=center}]

\node[main node] (1) {1234 \\ 2341 \\ 1232}; 
\node (7) [right of=1,node distance=1.1cm,transform canvas={shift={(0,0.5)}}] {``123''};

\node[main node] (2) [below left of=1] {4123};
\node (10) [left of=2,node distance=1.1cm] {``312''};
\node[main node] (3) [below right of=1] {3412};
\node (8) [right of=3,node distance=1.1cm] {``231''};

\node[main node] (13) [below left of=2] {1212 \\ 2321};
\node (15) [left of=13,node distance=1.1cm] {``121''};
\node[main node] (14) [below right of=3] {2121 \\ 2123};
\node (16) [right of=14,node distance=1.1cm] {``212''};

\node[main node] (4) [below right of=13] {2143};
\node (11) [left of=4,node distance=1cm] {``213''};
\node[main node] (5) [below left of=14] {1432};
\node (9) [right of=5,node distance=1cm] {``132''};

\node[main node] (6) [below left of=5] {4321 \\ 3214 \\ 3212};
\node (12) [left of=6,node distance=1.1cm,transform canvas={shift={(0,-0.4)}}] {``321''};

\path
(1) edge node {} (3);
\path
(2) edge node {} (1);
\path
(3) edge node {} (2);

\path
(6) edge node {} (4);
\path
(4) edge node {} (5);
\path
(5) edge node {} (6);

\path [draw,transform canvas={shift={(0,0.2)}}] 
(13) edge node {} (14);

\path [draw,transform canvas={shift={(0,-0.2)}}] 
(14) edge node {} (13);

\path
(1) edge [bend right=60] node {} (13);
\path
(13) edge [bend right=60] node {} (6);
\path
(6) edge [bend right=60] node {} (14);
\path
(14) edge [bend right=60] node {} (1);
     
\path
(1) edge [loop above] node {} (1);
\path
(6) edge [loop below] node {} (6);

\end{tikzpicture}}
\end{center}
\caption{Clustering the graph of overlapping 4-permutations that corresponds to the u-cycle 34321432345234}\label{dist-2-order-4}
\end{figure}

\section{Shortening u-cycles and u-words for permutations via usage of $\Diamond$s}\label{sec3}

In this section we consider the shortening problem via usage of $\Diamond$s. While the usage of the plain symbol $\Diamond$ seems to be dominated by various non-existence results (see Section~\ref{usage-diamonds}), the usage of $\Diamond_D$ may potentially result in interesting classification theorems, an example of which is given in Section~\ref{usage-diamond-a-b} (see Theorem~\ref{restr-Diamond}).

\subsection{Usage of $\Diamond$s}\label{usage-diamonds}

The following lemma is an analogue in the case of permutations of Theorem 4.1 in \cite{Goeckner-et-al} and Lemma 14 in \cite{CKMS} obtained for words. 

\begin{lemma}\label{structural}  Let $n\geq 3$ and $u = u_1u_2\cdots u_N$ be a u-p-cycle, or u-p-word, for $n$-permutations. If $u_k = \Diamond$ then $u_{k+n} =u_{k-n} =\Diamond$ assuming $k+n$ and/or $k-n$ exist in the case of u-p-words, and taking these numbers modulo $N$ in the case of u-p-cycles. \end{lemma}

\begin{proof} 

In what follows, the indices are taken modulo $N$ in the circular case. Suppose that $u_k = \Diamond$ and $u_{k+n}\neq \Diamond$. Further, suppose that $\pi=\pi_1\cdots \pi_{n-1}$ is 
one of the permutations obtained from $u_{k+1}\cdots u_{k+n-1}$ by substituting all the $\Diamond$s, if any, by any values and taking the reduced form. 

For the circular case, because $u_k=\Diamond$, the permutation $\pi$ cannot be covered by any other factor of $u$ (or else, some permutation ending with $\pi$ in the reduced form would be covered twice). However, this means that if $\pi$ is not monotone, at least one of the $n$-permutations $\red(\pi0)$ or $\pi n$ is not covered by $u$; contradiction. On the other hand, if $\pi$ is monotone, then we use the fact that $n\geq 3$, so even though both $\red(\pi0)$ and $\pi n$ can be covered by $u$, there is still at least one $n$-permutation not covered by $u$; contradiction.

For the non-circular case, there is a possibility for $\pi$ to occur one more time in $u$, namely, at its very beginning (that is, it is possible that $\red(u_1u_2\cdots u_{n-1})=\pi$). However, since $n\geq 3$, we know that  at least one of the $n$-permutations $\red(\pi0)$, 
$\red(\pi1^+)$ or $\pi n$ is not covered by $u$; contradiction.

One can use similar arguments, or use the fact that the reverse of a u-p-cycle/u-p-word is a u-p-cycle/u-p-word, to show that $u_{k-n} = \Diamond$.
\end{proof}

By the previous lemma, for any $\Diamond$ in a u-p-cycle or u-p-word $u$, the other two symbols in distance $n$ from it must be $\Diamond$s as well.
Thus the positions of $\Diamond$s are periodic in $u$ with period $n$, and any factor of $u$ of length $n$ contains equal number of
$\Diamond$s.
It follows that the notion of the {\em diamondicity} introduced next is well defined (see also \cite{Goeckner-et-al} where this notion
was introduced in the context of u-p-words over non-binary alphabets).
\begin{definition}
For a u-p-cycle or u-p-word $u$ for $n$-permutations, the diamondicity of $u$ is the number of $\Diamond$s
in any length $n$ factor in $u$.
\end{definition}

\subsubsection{U-p-cycles for permutations with $\Diamond$(s)}

Lemma~\ref{structural} yields Corollary~\ref{cor-u-cycle} below, which captures various rather restrictive conditions on relations between $n$ and $N$ to 
be satisfied by any u-p-cycle for permutations. In the proof of Corollary~\ref{cor-u-cycle}, we need the following simple number theoretical fact as well as the well known result stated in Lemma~\ref{FWPL}.

\begin{lemma}\label{number_t}
If $n$ and $N$ are two positive integers, $c=\gcd(n,N)$, and $I=\left\{0,1,\ldots,\frac{N}{c}-1\right\}$, then 
$$
\left\lbrace i\cdot \frac{n}{c}\mod \frac{N}{c}\,:\,i\in I\right\rbrace=I.
$$
\end{lemma}
\begin{proof}
We show that the integers of the form  $i\cdot \frac{n}{c}\mod \frac{N}{c}$, $i\in I$, are all different.
If $i,i'\in I$ with $i\neq i'$, then 
$i\cdot \frac{n}{c}\mod \frac{N}{c}\neq i'\cdot \frac{n}{c}\mod \frac{N}{c}$.
Indeed, otherwise  $(i-i')\cdot \frac{n}{c}$ is a multiple of $\frac{N}{c}$, or equivalently 
$(i-i')\cdot n$ is a common multiple of $n$ and $N$, which yields a contradiction
since $|i-i'|<\frac{N}{c}$ and $\mathrm{lcm}(n,N)=\frac{N}{c}\cdot n$.
Thus, the sets in question have the same cardinality, which completes the proof.
\end{proof}

\begin{lemma}[Fine and Wilf's periodicity lemma, \cite{FW}]\label{FWPL} Any word having periodicities $p$ and $q$ and length $\geq p+q-\gcd(p,q)$ has periodicity $\gcd(p, q)$. \end{lemma}

\begin{corollary}\label{cor-u-cycle} Let $u = u_1u_2\cdots u_N$ be a u-p-cycle (with or without $\Diamond$(s)) for $n$-permutations. Then we have 
\begin{itemize}
\item[(i)] $N =k!$, where $n-k$ is the diamondicity  of $u$.
\end{itemize}

\noindent
In addition, if $c=\gcd(n,N)$, then 
\begin{itemize}
  \item[(ii)] the occurrences of $\Diamond$s in $u$ are $c$-periodic, and
  \item[(iii)] $\frac{n}{c}$ divides $n-k$, so $c\neq 1$ for $1\leq k\leq n-1$. 
\end{itemize}

\end{corollary}

\begin{proof} {\it (i)}
The number of $\Diamond$s in each factor of $u$ of length $n$ is $n-k$, and thus such a factor covers $${n\choose k}(n-k)!=\frac{n!}{k!}$$ permutations of
length $n$, and there must be $k!$ length $n$ factors (read cyclically) to cover all $n!$ permutations. 

\noindent
{\it (ii)} The statement follows by Lemma~\ref{structural} and Lemma~\ref{FWPL} applied to $n$ and $N$. However, we provide an alternative proof here.

Factoring $u$ as $\underbrace{u_1u_2\cdots u_c}_{v_1}\ \underbrace{u_{c+1}u_{c+2}\cdots u_{2c}}_{v_2}\ \cdots\ 
\underbrace{u_{N-c+1}u_{N-c+2}\cdots u_{N}}_{v_\frac{N}{c}}$, we have $u = v_1v_2\cdots v_\frac{N}{c}$ where  $v_i$, $1\leq i\leq \frac{N}{c}$,
is the length $c$ factor $u_{c\cdot(i-1)+1} u_{c\cdot(i-1)+2}\cdots u_{c\cdot i}$. With this notation, it follows that the number of $\Diamond$s in 
$v_1$ is the same as that in $v_{\frac{n}{c}+1}$. Indeed,
$v_1v_2\cdots v_{\frac{n}{c}}$ and $v_2v_3\cdots v_{\frac{n}{c}+1}$ are two length $n$ 
factors of $u$ which overlap when $c\neq n$,
and by Lemma~\ref{structural} they have the same number of $\Diamond$s, and so do 
$v_1$ and $v_{\frac{n}{c}+1}$.
Similarly, and taking the indices modulo $\frac{N}{c}$, the length $c$ factors $v_{\frac{n}{c}+1}$ and $v_{\frac{2n}{c}+1}$
have the same number of $\Diamond$s. And generally, each of the length $c$ factors 
$v_{\frac{i\cdot n}{c}+1}$, $0\leq i<\frac{N}{c}$, has the same number of  $\Diamond$s.
By Lemma \ref{number_t}, the set $\left\{i\cdot \frac{n}{c}+1\,:\,0\leq i<\frac{N}{c}\right\}$ is precisely
the set $\left\{i\,:\,1\leq i\leq\frac{N}{c}\right\}$, and thus each of the length $c$ factors $v_i$, $1\leq i\leq\frac{N}{c}$, 
has the same number of  $\Diamond$s.

Clearly, $u_2u_3\cdots u_Nu_1$ is a u-p-cycle for $n$-permutations too, and factoring it as
$\underbrace{u_2u_3\cdots u_{c+1}}_{v'_1}\ \underbrace{u_{c+2}u_{c+3}\cdots u_{2c+1}}_{v'_2}\ \cdots\ 
\underbrace{u_{N-c+2}u_{N-c+3}\cdots u_{1}}_{v'_\frac{N}{c}}$, and reasoning as previously, we have that 
each $v'_i$ has the same number of $\Diamond$s (which is the same as that of $v_i$s).
Repeating this process, we have finally that each length $c$ factor of $u$ has the same number of $\Diamond$s.

\noindent
{\it (iii)}
By (ii) it follows that 
$\frac{n}{c}$ divides the number of $\Diamond$s in each factor of length~$n$.
\end{proof}

The following corollary of Corollary \ref{cor-u-cycle} refines Lemma~\ref{structural} in the case of u-p-cycles for permutations. 

\begin{corollary}
With the notations in Corollary~\ref{cor-u-cycle}, if $u$ is a u-p-cycle for $n$-permutations,
 the positions of $\Diamond$s in $u$ are periodic with period~$c$.
\end{corollary}

In the next corollary, we give two proofs for the case when $n$ is a prime number. 

\begin{corollary}\label{cor-10} If $n$ is a prime number, or $n=4$, then there exists no u-p-cycle  for $n$-permutations. \end{corollary}
\begin{proof}
If $n=4$, it follows from Corollary \ref{cor-u-cycle}
that the admissible values of $N$ are $2$ and $6$, corresponding to $k=2,3$, respectively.
Clearly only $N=6$ can be the length of a u-p-cycle, thus  
$k=3$ and $c=\gcd(n,N)=2$. By (iii) in Corollary \ref{cor-u-cycle}, $\frac{n}{c}=2$ divides $n-k=1$, contradiction.

If $n$ is prime, from (iii) in Corollary \ref{cor-u-cycle} (and with the notations therein) 
$\gcd(n,N)=n$, which contradicts (i) in Corollary \ref{cor-u-cycle}, namely, that $n$ divides $N=k!$, with  $k<n$.

An alternative proof for the case when $n$ is prime is as follows. The total number of $\Diamond$s counted in {\em all} factors of length $n$ is $k!(n-k)$. However, each $\Diamond$ was counted exactly $n$ times, so $n$ must divide $k!(n-k)$, which is impossible if $n$ is a prime number since $k<n$.\end{proof}

We conclude the subsection with two more non-existence results, the first of which is also applicable to u-p-words for permutations to be considered in the next subsection. Recall that  by Lemma~\ref{structural}, $\Diamond$s in a u-p-word or a u-p-cycle must occur periodically. 

\begin{theorem}\label{period-2-diamonds} 
For any $n\geq 1$, there are no non-trivial u-p-words or u-p-cycles for $n$-permutations in which $\Diamond$s occur periodically with period~$2$. \end{theorem}

\begin{proof} Suppose that such a u-p-word, or u-p-cycle $u=u_1u_2\cdots u_N$ for permutations exists, where $N\geq n+1$ because $u$ is non-trivial. Then $u_1u_2\cdots u_n$ is of one of the following four forms: 
\begin{enumerate}
\item $u_1\Diamond u_3\Diamond u_5\cdots u_n$;
\item $u_1\Diamond u_3\Diamond u_5\cdots \Diamond$;
\item $\Diamond u_2\Diamond u_4 \Diamond \cdots u_n$;
\item $\Diamond u_2\Diamond u_4 \Diamond \cdots \Diamond$.
\end{enumerate}
In either case, we claim that there exists an $n$-permutation that is covered by both $u_1u_2\cdots u_n$ and $u_2u_3\cdots u_{n+1}$, contradicting $u$'s properties. Next we provide such permutations for the first two cases; the remaining two cases are similar and their considerations are omitted.
\begin{enumerate}
\item $u_1a_3 u_3a_5 u_5\cdots a_nu_n$, where $\red(a_3a_5\cdots a_n)=\red(u_3u_5\cdots u_n)$ and each of $a_i$s is larger than any $u_j$ (clearly, the $\Diamond$s can be assigned in such values). This permutation is also covered by $u_2u_3\cdots u_{n+1}$ by choosing the values of the $\Diamond$s from left to right to be $b_3b_5\cdots b_{n+2}$, such that $\red(b_3b_5\cdots b_{n+2})=\red(u_1u_3\cdots u_n)$, and  each of $b_i$s is smaller than any $u_j$.

\item $u_1a_1 u_3a_3 \cdots u_na_n$, where $\red(a_1a_3\cdots a_n)=\red(u_1u_3\cdots u_n)$ and each of $a_i$s is larger than any $u_j$. This permutation is also covered by $u_2u_3\cdots u_{n+1}$ by choosing the values of the $\Diamond$s from left to right to be $b_3b_5\cdots b_{n+1}$, such that $\red(b_3b_5\cdots b_{n+1})=\red(u_1u_3\cdots u_{n-1})$, and  each of $b_i$s is smaller than any $u_j$.
\end{enumerate}
\end{proof}

\begin{theorem}\label{period-3-diamonds} For any $n\geq 1$, there are no non-trivial u-p-cycles for $n$-permutations in which $\Diamond$s occur periodically with period~$3$.  \end{theorem}

\begin{proof} 
Suppose that such a u-p-cycle $u$ exists. Note that 3 must divide $n$.  Indeed, if 3 does not divide $n$, then we can connect any pair of positions in $u$ cyclically with steps of length 3, so by Corollary~\ref{cor-u-cycle}, all symbols would have to be $\Diamond$s, making $u$ trivial and contradicting the assumption. Thus, we have three cases to consider based on which factor covers the increasing $n$-permutation. In each of the cases it is crucial that our universal word $u$ is cyclic, because we do not know the location of the factor covering the increasing permutation. Without loss of generality, we assume that in the factor covering the increasing permutation, the non-$\Diamond$ symbols are $1,2,3,\ldots$.

\begin{itemize}
\item The increasing permutation is covered by the factor $$\Diamond 12 \Diamond 34\cdots \Diamond \left(\frac{2n}{3}-1\right)\frac{2n}{3}.$$ Then this permutation is covered one more time starting from the letter 1, since the value of the $\Diamond$ next to $\frac{2n}{3}$ (cyclically) can be chosen $\left(\frac{2n}{3}+1\right)$; contradiction.
\item The increasing permutation is covered by the factor $$12\Diamond 34 \Diamond \cdots  \left(\frac{2n}{3}-1\right)\frac{2n}{3}\Diamond.$$ Picking the value of the $\Diamond$ immediately to the left (cyclically) of the letter 1 to be $1^-$ we see that the increasing permutation is covered one more time starting from this position; contradiction.  
\item The increasing permutation is covered by the factor \begin{equation}\label{eq-1} 1\Diamond 2 3\Diamond 4\cdots \left(\frac{2n}{3}-1\right)\Diamond \frac{2n}{3}.\end{equation} 
Consider the factor 
\begin{equation}\label{eq-2} 2 3\Diamond 4\cdots \left(\frac{2n}{3}-1\right)\Diamond \frac{2n}{3} x\Diamond \end{equation}  of length $n$, where $x$ is some letter. No matter what $x$ is, we cover some permutation (not necessarily increasing) twice, which leads to a contradiction. Indeed, the rightmost $\Diamond$ in (\ref{eq-2}) can be chosen to be maximum in the permutation, while the rightmost $\Diamond$ in (\ref{eq-1}) can be chosen to be equivalent to $x$ in (\ref{eq-2}).
\end{itemize}
\end{proof}

\begin{remark}
{\em Unfortunately, the arguments in Theorems~\ref{period-2-diamonds} and \ref{period-3-diamonds} do not seem to be possible to extend to periods of length $4$, or more. }
\end{remark}

\subsubsection{U-p-words for permutations with $\Diamond$(s)}

Clearly, $\Diamond$ and $\Diamond 1$ are, respectively, u-p-words for the 1-permutation and 2-permutations. The following proposition shows that these are the only u-p-words with a single $\Diamond$ placed at the beginning of the word. Before stating the proposition, we introduce a notion related to the clustered graph of overlapping permutations that will be used in some of our proofs. 

\begin{definition}\label{reach} Let $u_iu_{i+1}\cdots u_{i+n}$ be a factor of a u-p-word $u_1u_2\cdots u_N$ for $n$-permutations. We say that the edge coming out from the permutation $\red(u_iu_{i+1}\cdots u_{i+n-1})$ in the clustered graph of overlapping permutations is used to {\em reach} the permutation $\red(u_{i+1}\cdots u_{i+n})$.
\end{definition}

\begin{proposition}\label{diamond-at-first} Let $n\geq 3$. No u-p-word for $n$-permutations with a single $\Diamond$ of the form $u=\Diamond u_2u_3\cdots u_{N}$ exists. \end{proposition}

\begin{proof} Since $n\geq 3$, it is clear that $N\geq n+1$. We can now apply Lemma~\ref{structural} to obtain the desired result. \end{proof}

The case $n=4$ in the next proposition follows from our more general Corollary~\ref{nice-corol} below. However, we keep this case in Proposition~\ref{prop1} for yet another illustration of our straightforward approach to prove some of the non-existence statements.

\begin{proposition}\label{prop1}
For $n=3,4$ there is no u-p-word for $n$-permutations with a single $\Diamond$ of the form $u=u_1\Diamond u_3u_4\cdots u_N$.  
\end{proposition}

\begin{proof} Let $n=3$. Without loss of generality (using the complement operation, if necessary), we can assume that $u$ begins with $1\Diamond 2$. Then the possible continuations of $u$ are $1\Diamond 22^+$, $1\Diamond 22^-$ and $1\Diamond 21^-$. But then the following permutations are covered twice, respectively, 123, 132 and 132. 

Let $n=4$. Without loss of generality (using the complement operation, if necessary), we can assume that there are three cases of beginning of $u$ to consider.

\begin{itemize}
\item $1\Diamond 23$. Possible continuations are as follows.
\begin{itemize}
\item $1\Diamond 234$. The permutation 1234 is covered twice; contradiction.
\item $1\Diamond 232^+x$ for some $x$. Note that so far three permutations, namely,1324, 1423, and $\red(232^+x)$ from the cluster with the signature ``132'' were covered. But the fourth permutation from that cluster will never be covered (or else, because of $\Diamond 232^+$, some permutation ending with the pattern $132$ will be covered twice).
\item $1\Diamond 231x$ for some $x$.  Because of the factor $\Diamond 231$, the permutation $\red(231x)$ will be the only one covered in the cluster with the signature ``231'' (no such permutation can be covered starting at the leftmost position, or at the $\Diamond$). Contradiction with the cluster having four permutations.

\end{itemize}
\item $1\Diamond 32$. 
\begin{itemize}
\item $1\Diamond 324x$ for some $x$. Note that so far three permutations, namely, 2143, 3142, and $\red(324x)$ from the cluster with the signature ``213'' were covered. But the fourth permutation from that cluster will never be covered (or else, because of $\Diamond 324$, some permutation ending with the pattern $213$ will be covered twice).
\item $1\Diamond 322^+x$ for some $x$. Because of the factor $\Diamond 322^+$, the permutation $\red(322^+x)$ will be the only one covered in the cluster with the signature ``312'' (no such permutation can be covered starting at the leftmost position, or at the $\Diamond$). Contradiction with the cluster having four permutations.
\item $1\Diamond 321$. The permutation 1432 is covered twice; contradiction.
\end{itemize}
\item $2\Diamond 13$. 
\begin{itemize}
\item $2\Diamond 134$. The permutation 3124 is covered twice; contradiction.
\item $2\Diamond 132x$ for some $x$. Because of the factor $\Diamond 132$, the permutation $\red(132x)$ will be the only one covered in the cluster with the signature ``132'' (no such permutation can be covered starting at the leftmost position, or at the $\Diamond$). Contradiction with the cluster having four permutations.
\item $2\Diamond 131^-x$ for some $x$. Note that so far three permutations, namely, 2314, 2413, and $\red(131^-x)$ from the cluster with the signature ``231'' were covered. But the fourth permutation from that cluster will never be covered (or else, because of $\Diamond 131^-$, some permutation ending with the pattern $231$ will be covered twice.
\end{itemize}
\end{itemize}
\end{proof}

%

Let $u$ be a u-p-word for $n$-permutations with diamondicity $d$. 
It follows (see also the proof of the first part of Corollary \ref{cor-u-cycle}) that $u$ must contain
exactly $(n-d)!$ different factors, and thus the length of $u$ is $(n-d)!+n-1$.

\begin{theorem}\label{constraint_n}
Let $u$ be a non-trivial u-p-word for $n$-permutations, and let $f$ be the number of $\Diamond$s in $u$. 
Then $n\leq 3f +1$.
\end{theorem}
\newpage
\begin{proof}
Let $d\geq 1$ be the diamondicity of $u$. Thus, the length of $u$ is $(n-d)!+n-1$, and
the number $f$ of $\Diamond$ symbols in $u$ satisfies: 
\begin{eqnarray*}
f & \geq &  \left\lfloor\frac{(n-d)!+n-1}{n}\right\rfloor\cdot d\\
  & =    &  \left\lceil\frac{(n-d)!}{n}\right\rceil\cdot d\\
  & \geq &  \frac{(n-d)!}{n}\cdot d\\
  & \geq & \frac{(n-d)\cdot (n-(d+1))}{n}\\
  & \geq &  n-(2d+1).
\end{eqnarray*}
It follows that $n\leq f+2d+1\leq 3f+1$, and the statement holds.
\end{proof}
As is mentioned above, $\Diamond$ is a trivial u-p-word for the 1-permutation, and $\Diamond 1$
is a u-p-word for $2$-permutations. These are the only u-p-words for permutations with a single $\Diamond$ as shown by the following corollary.

\begin{corollary}\label{nice-corol}
For $n\geq 3$, there is no u-p-word for $n$-permutations with a single $\Diamond$.
\end{corollary}
\begin{proof}
By Theorem~\ref{constraint_n}, if $u$ is a u-p-word for $n$-permutations with a single $\Diamond$, then 
$n\leq 4$.

Using the reverse operation, if necessary, one can assume that the single $\Diamond$ in a u-p-word for permutations is in its first half. Thus, by Propositions~\ref{diamond-at-first} and~\ref{prop1}, no u-p-word exists for $3$-permutations. 

If $n=4$, then by Lemma~\ref{structural}, since we have exactly one $\Diamond$, the length of a u-p-word must be at most 7. On the other hand, this length must be $(n-d)!+n-1=(4-1)!+4-1=9$; contradiction. 
\end{proof}

\subsection{Usage of $\Diamond_{a,b}$}\label{usage-diamond-a-b}

Recall that $\Diamond_{a,b}$, where $a,b\in\{1,2,\ldots,n\}$, $a<b$, denotes the set of permissible substitutions in an $n$-permutation. For example, $a=1$ allows usage of the smallest element in all factors containing $\Diamond_{a,b}$, while $a=2$ allows usage of the next smallest element in these factors, and so on. 

\begin{theorem}\label{restr-Diamond} Let $n\geq 2$ and $a<b$. Then, necessary and sufficient conditions for existence of a u-p-word for $n$-permutations of the form $\Diamond_{a,b} u_2u_3\cdots u_{N}$ are  
\begin{itemize} 
\item $a=1$ and $\red(u_2u_3\cdots u_{n})=12\cdots (n-1)$, or 
\item $b=n$ and $\red(u_2u_3\cdots u_{n})=(n-1)(n-2)\cdots 1$.
\end{itemize}
\end{theorem}

\begin{proof} Similarly to the proof of Lemma~\ref{structural}, consider the cluster $C$ corresponding to the signature ``$\red(u_2u_3\cdots u_n)$''. If $u_2u_3\cdots u_{n}$ is not monotone (increasing or decreasing) then because of the factor $\Diamond_{a,b} u_2u_3\cdots u_{n}$ we see that reaching the permutation $\red(u_2u_3\cdots u_{n+1})$ in $C$ (recall Definition~\ref{reach}) uses two edges, so that at least one permutation in $C$ will never be covered by $u$. On the other hand, one can see that exactly the same situation occurs if $\red(u_2u_3\cdots u_{n})=12\cdots (n-1)$ and $a\neq 1$ (if $a=1$ then one of the two edges mentioned above is a loop and there is no contradiction), and if $\red(u_2u_3\cdots u_{n})=(n-1)(n-2)\cdots 1$ and $b\neq n$ (again, if $b=n$ then one of the two edges is a loop giving no contradiction). 

On the other hand, if one of the two conditions are satisfied, then we have that the $\Diamond_{a,b}$ is responsible for removing an edge coming to the respective cluster $C$ with a monotone signature and the loop connected to $C$ from the clustered graph of overlapping permutations, as well as covering two permutations, one from $C$ and one from another cluster $C'$. The rest of the word $u_2u_3\cdots u_N$ corresponds to an Eulerian path beginning at $C$ and ending at $C'$, which exists because each cluster is balanced, except for $C$ (one extra out-edge) and $C'$ (one extra in-edge), and the graph is clearly still strongly connected. \end{proof}

An example of a u-word for $3$-permutations given by Theorem~\ref{restr-Diamond} is $\diamond_{1,5}243241$.

\section{Concluding remarks}\label{final-sec} 

This paper opens up a new research direction of shortening u-cycles and u-words for permutations that naturally extends analogous studies conducted for the celebrated de Bruijn sequences \cite{CKMS,Goeckner-et-al}. We were able to offer two different ways to approach the problem, namely, via linear extensions of posets, and via usage of (restricted) $\Diamond$s, and we discussed several existence and non-existence results related to the context.

Out of possible directions for further research, it would be interesting to prove, or disprove, Conjecture~\ref{conj1} and our guess that no distribution of $\Diamond$s in a word can result in a non-trivial construction of a u-p-word for $n$-permutations for $n\geq 3$. Moreover, it would be interesting to  extend the results in Section~\ref{possible-lengths-u-cycles-sec} to the case of $n$-permutations for $n\geq 5$, namely, to answer the following question:  Do there exist  u-words and u-cycles for $n$-permutations of length larger than $n$ but smaller than $n!-(n-1)!+n-1$ and $n!-(n-1)!$, respectively? Also, we would like to see some characterization theorems involving (more than one) $\Diamond_D$ for $D$ not necessarily of size $2$, thus extending the result of Theorem~\ref{restr-Diamond}. 

Some enumerative questions can be raised as well. For example, one should be able to count u-words of various lengths in  Theorem~\ref{thm1}, that should be based on the choice of $k$ cycles formed by double edges to be replaced by single edges (out of the total number of $(n-2)!$ cycles formed by double edges).

Finally, note that there are other ways to define  the concept of a universal cycle/word for permutations. For example, in \cite{NF} permutations are listed as consecutive substrings without using the notion of order-isomorphism. The ideas in this paper for shortening u-cycles/u-words for permutations can be used for other contexts, such as those in \cite{NF}.


\section*{Acknowledgments}

The authors are grateful to Sergey Avgustinovich for useful discussions related to the paper. Also, the first author is thankful to the Universit\'e de Bourgogne for its hospitality during the author's visit of it in November 2016. The second author is grateful to the Presidium of the Russian Academy of Sciences for supporting his research (Grant 0314-2015-0011). Finally, the authors are grateful to  the anonymous referees for reading carefully the manuscript and for providing many useful suggestions.

\end{document}